\documentclass[12pt, reqno, a4paper]{amsart}

\usepackage{hyperref, amssymb, amsmath, enumerate, amsfonts, amsthm, mathrsfs, url, bm}
\usepackage{atbegshi,picture}
\usepackage{fullpage}

\usepackage{color}

\usepackage{mystyle}

\makeatletter
\let\@@pmod\pmod
\DeclareRobustCommand{\pmod}{\@ifstar\@pmods\@@pmod}
\def\@pmods#1{\mkern4mu({\operator@font mod}\mkern 6mu#1)}
\makeatother

\begin{document}

\title[The transference principle]{Four variants of the Fourier-analytic transference principle
}

\author{Sean Prendiville}


\thanks{The author was supported by Tim Browning's ERC grant \texttt{306457}.  Thanks to Tim for many useful conversations.}

\address{S.M.~Prendiville, Mathematical Institute, University of Oxford, Oxford, OX2 6GG, United Kingdom}
\email{sean.prendiville@gmail.com}

\date{\today}

\begin{abstract}
We survey four instances of the Fourier analytic `transference principle' or `dense model lemma', which allows one to approximate an unbounded 
function on the integers by a bounded function with similar Fourier transform.  Such a result forms a component of a general method pioneered by Green  to count solutions to a single linear equation in a sparse subset of integers.  
\end{abstract}

\maketitle

\setcounter{tocdepth}{1}
\tableofcontents
\thispagestyle{empty}

\section{Introduction}

\subsection{Aim}

There has been much recent work on counting arithmetic configurations in a sparse set of integers, such as the set of primes \cite{greenprimes, greentaoprimes, greentaolinear}, smooth numbers \cite{harper}, random sets \cite{conlongowers, schacht}, pseudorandom sets \cite{conlonfoxzhao}, or dense subsets thereof.  Given such a sparse set, it is often useful to be able to construct a dense subset of integers whose arithmetic properties resemble those of the sparse set, the theory being much more developed in the dense regime, with recourse to powerful results such as Szemer\'edi's theorem and affiliated techniques.  

When counting solutions to a single linear equation, the arithmetic closeness of the dense model set to our original sparse set can be measured by the level of similarity in their Fourier transform, provided that we weight the characteristic function of our sparse set suitably.  The sparseness of our set forces this weight function to grow asymptotically, so we are left with the problem of approximating an unbounded function by a bounded function, with the closeness of approximation measured by the $L^\infty$-norm of their Fourier transform.  The purpose of this note is to survey four variants of such a \emph{bounded approximation lemma}, also called a \emph{transference principle} or \emph{dense model lemma} in the literature: the original found in Green \cite{greenprimes}, a quantitative improvement due to Helfgott--De Roton \cite{helfgottderoton}, a further quantitative refinement due to Naslund \cite{naslund}, and finally a much more general technique due (independently) to Gowers \cite{gowers10} and Reingold et al \cite{RTTVnote}.  Our focus is on the quantitative strength of each of these results.  We give a complete account of the required background in the appendices.

The Fourier-analytic transference principle is particularly powerful when combined with the Hardy--Littlewood circle method.  Traditionally, the circle method is performed with respect to a function defined on the integers, whose Fourier transform is then defined on the circle group $\T = \R/\Z$.  In the majority of the references we survey, the Fourier analysis is performed with respect to functions defined on the integers modulo a large prime number.  This has the expositional advantage that both physical and phase space are discrete, and in fact isomorphic.  However, we believe this reduction is artificial, and in order to highlight the utility of the transference principle within the traditional number-theoretic circle method, we opt to give all proofs with respect to Fourier analysis on the integers. 

\subsection{Motivation: a sparse version of Roth's theorem}

A theorem of Roth \cite{roth53, roth53II} quantifies the density required of a set of integers to ensure that it contains a non-trivial solution to a single linear equation
\begin{equation}\label{linear equation}
c_1 x_1 + \dots + c_s x_s = 0.
\end{equation}  Assuming the coefficients sum to zero, a variant of this theorem due to Bloom \cite{bloom12} states that for any $\delta > 0$ there exists $c(\delta) > 0$ such that if $A$ is a subset of $[N]$ of density at least $\delta$ (i.e.\ $|A| \geq \delta N$), then $A$ contains many solutions to the equation, in that
\begin{equation}\label{lower bound}
\sum_{\vc\cdot\vx = 0} 1_A(x_1) \dotsm 1_A(x_s) \geq c(\delta) N^{s-1},
\end{equation}
where one may take
\begin{equation}\label{bloom dependence}
c(\delta) \gg_{\vc} \exp\brac{-C/\delta^{\frac{1}{s-2-\eps}}}
\end{equation}
for some absolute constant $C = C(s, \eps)$ and any $\eps > 0$.

Roth's method for proving such a result proceeds by exploiting the orthogonality relation
\begin{equation}\label{orthogonality}
\sum_{\vc\cdot\vx = 0} 1_A(x_1) \dotsm 1_A(x_s) = \int_\T \hat{1}_A(c_1\alpha) \dotsm\hat{1}_A(c_s\alpha)\intd \alpha,
\end{equation}
where we define the Fourier transform of a function $f : \Z \to \C$ of finite support by
\begin{equation}\label{fourier transform}
\hat{f}(\alpha) := \sum_n f(n) e(\alpha n).
\end{equation}
If we know the distribution of $A$ in arithmetic progressions, then the classical circle method allows us predict the behaviour of $\hat{1}_A$.  Roth's argument says that either $A$ is equidistributed in arithmetic progressions, in which case we can calculate $\hat{1}_A$ and therefore \eqref{orthogonality}, or alternatively $A$ is biased towards at least one arithmetic progression.  Exploiting this bias then forms the so-called `density increment' argument.  We refrain from the details here, but hope to convey to the reader the sense that if one knows how the Fourier transform $\hat{1}_A$ behaves, then one can count solutions to a linear equation in $A$.

Suppose that we wish to prove an analogue of Roth's theorem for subsets of the integers which are not dense in the interval $[N]$, but are dense in some fixed sparse subset $S \subset [N]$, so that $|S| = o(N)$.  For example, one may take $S$ to be of arithmetic nature, such as the set of primes or the set of squares, or even the set of squares of primes.  Alternatively, one could take $S$ to be a random subset of $[N]$.

Given a subset $A$ of our sparse set $S$ which is \emph{relatively dense}, in the sense that $|A| \geq \delta |S|$, we wish to prove a lower bound of the form \eqref{lower bound}.  Notice that if we could construct a \emph{dense} subset $B \subset [N]$, with $|B| \geq \delta^c N$ say, and such that we have the Fourier approximation
\begin{equation}\label{dense approximation}
\hat{1}_A \approx \hat{1}_B,
\end{equation}
then we can employ Roth's theorem to obtain a lower bound for the number of solutions to \eqref{linear equation} in $A$ as follows
\begin{equation}\label{fourier approx count}
\begin{split}
\sum_{\vc\cdot\vx = 0} 1_A(x_1) \dotsm 1_A(x_s) & = \int_\T \hat{1}_A(c_1\alpha) \dotsm\hat{1}_A(c_s\alpha)\intd \alpha\\
& \approx \int_\T \hat{1}_B(c_1\alpha) \dotsm\hat{1}_B(c_s\alpha)\intd \alpha\\
& = \sum_{\vc\cdot\vx = 0} 1_B(x_1) \dotsm 1_B(x_s)\\
& \geq c(\delta^c) N^{s-1}.
\end{split}
\end{equation}

Obtaining an approximation such as \eqref{dense approximation} is the strategy of the Fourier-analytic transference principle, or dense model lemma, originating in Green \cite{greenprimes}.  However, as stated, such an approximation is too much to hope for.  Looking at the Fourier transform evaluated at $\alpha = 0$, we would deduce that
\begin{equation}\label{zero fourier coefficient}
|S| \geq |A| = \hat{1}_A(0) \approx \hat{1}_B(0) \geq \delta^c N.
\end{equation}
This would then imply that $S$ is itself a dense subset of the interval $[N]$.  To get around this, we weight the indicator function of $A$ in order to ensure that we have some hope of approximating its Fourier transform by the Fourier transform of a dense set.  In order to deal with arbitrary relatively dense subsets of $S$, it makes sense to choose this weighting independently of the set $A$ itself.  
\begin{definition}[Majorant]
Given $S \subset [N]$ define a \emph{majorant on $S$} to be a non-negative function $\nu : \Z \to [0, \infty)$ with support contained in $S$ and such that
\begin{equation}\label{L1 assumption}
\sum_n \nu(n) = \bigbrac{1 + o(1) }N.
\end{equation}
\end{definition}

Given a majorant $\nu$ on $S$ and $A \subset S$ with relative density $|A| \geq \delta |S|$, define
$$
f := 1_A \nu.
$$
Provided that we choose our majorant sensibly, we should be able to prove that
\begin{equation}\label{original density}
\sum_n f(n) \geq \delta^c N
\end{equation}
for some $c > 0$.
We therefore hope to obtain a Fourier approximation of the form
$$
\hat{f} \approx \hat{1}_B,
$$
for some $B \subset [N]$ which is suitably dense $|B| \geq \delta^c N$.  In fact, we do not need our dense approximant to be the characteristic function of a set; it suffices for the function to have bounded $L^2$-norm, an observation first recorded by Helfgott and De Roton \cite{helfgottderoton}.

\begin{lemma}[$L^2$-boundedness suffices]\label{L2 suffices}  Let $c_1 + \dots + c_s = 0$.  Then for any $\delta > 0$ and any constant $C$ there exists $c(\delta, C) > 0$ such that the following holds.
Suppose that $g:\Z \to [0, \infty)$ is a non-negative function supported on $[N]$ which has bounded\footnote{Although this estimate for the $L^2$-norm appears to grow with $N$, it is the same estimate one would obtain for a bounded function on $[N]$.}  $L^2$-norm
\begin{equation}\label{L2 assumption}
\sum_n g(n)^2 \leq C N.
\end{equation}
Then the density assumption
\begin{equation}\label{density assumption}
\sum_n g(n) \geq \delta N
\end{equation}
implies that
$$
\sum_{\vc\cdot \vx = 0} g(x_1) \dotsm g(x_s) \geq c(\delta, C) N^{s-1}.
$$
In fact, one may take 
\begin{equation}\label{c dependence}
c(\delta, C) = (\delta/2)^s c\brac{\tfrac{\delta^2}{4C}},
\end{equation}
where $c(\delta)$ is the constant appearing in Roth's theorem \eqref{lower bound}.
\end{lemma}

\begin{proof}
Define
$$
B := \set{x \in [N] : g(x) \geq \delta/2}.
$$
Then, employing the Cauchy--Schwarz inequality, we have
\begin{equation}\label{cauchy employment}
\begin{split}
\delta N \leq \sum_x g(x)  &= \sum_{x \notin B} g(x) + \sum_{x \in B} g(x)\\
&\leq \trecip{2}\delta N  + |B|^{1/2}\brac{\sum_x  g(x)^2}^{1/2}\\
& \leq \trecip{2}\delta N + \brac{|B|CN}^{1/2}.
\end{split}
\end{equation}
Therefore
\begin{equation}\label{large value density}
|B| \geq \frac{\delta^2}{4C} N.
\end{equation}

Applying Bloom's variant of Roth's theorem, we deduce that
$$
\sum_{\vc\cdot\vx = 0} 1_B(x_1) \dotsm 1_B(x_s) \geq c\brac{\tfrac{\delta^2}{4C}} N^{s-1}.
$$
Hence
\begin{align*}
\sum_{\vc\cdot \vx} g(x_1) \dotsm g(x_s) \geq (\delta/2)^s\sum_{\vc\cdot \vx} 1_B(x_1) \dotsm 1_B(x_s) \geq (\delta/2)^s c\brac{\tfrac{\delta^2}{4C}} N^{s-1}.
\end{align*}
\end{proof}

Let us sketch how this result, when combined with a transference principle, allows one to extract a quantitative bound on the relative density of a subset $A \subset S \subset [N]$ lacking non-trivial solutions to \eqref{linear equation}.  Write $\delta := |A|/|S|$ for the relative density of $A$ in $S$.  Then provided that one has made a sensible choice for the weighted majorant $\nu$ on $S$, one should have
$$
\sum_n 1_A(n) \nu(n) \geq \delta^c N,
$$
for some absolute $c > 0$.  Applying a transference principle to the function $f = 1_A \nu$, one obtains an approximant $g$ supported on $[N]$ with bounded $L^2$-norm of the form \eqref{L2 assumption} and such that $\hat{f} \approx \hat{g}$ uniformly on $\T$.  Performing an approximation similar to \eqref{fourier approx count} and applying Lemma \ref{L2 suffices} yields
\begin{equation}\label{f lower bound}
\sum_{\vc \cdot \vx = 0} f(x_1) \dotsm f(x_s) \geq c(\delta^c, C) N^{s-1}.
\end{equation}
Yet if $A$ contains only trivial solutions to \eqref{linear equation}, we have
\begin{equation}\label{f nu count bound}
\sum_{\vc \cdot \vx = 0} f(x_1) \dotsm f(x_s) \leq \sum_{\substack{\vc \cdot \vx=0\\ \vx \text{ trivial}}} \nu(x_1) \dotsm \nu(x_s).
\end{equation}

There are various possible candidates for what should constitute a trivial solution to \eqref{linear equation}, one such choice being that $\vx$ belongs to one of a finite collection of proper subspaces of the hyperplane $\vc\cdot\vx = 0$.  Whatever choice of triviality one makes, one would expect that the trivial solutions should be a sparse subset of the solution space, so that
$$
\sum_{\substack{\vc \cdot \vx=0\\ \vx \text{ trivial}}} 1_{[N]}(x_1) \dotsm 1_{[N]}(x_s) \leq \frac{N^{s-1}}{\omega(N)}
$$
for some function $\omega(N) \to \infty$.
Moreover, a sensible choice of majorant should respect this sparseness, so that
$$
\sum_{\substack{\vc \cdot \vx=0\\ \vx \text{ trivial}}} \nu(x_1) \dotsm \nu(x_s) \ll  \frac{N^{s-1}}{\omega(N)}.
$$
Combining this with \eqref{f lower bound} and \eqref{f nu count bound} yields
$$
c(\delta, C) \ll \recip{\omega(N)}.
$$
Using the lower bounds \eqref{bloom dependence} and \eqref{c dependence} then allows us to extract an upper bound on $\delta$ in terms of $\omega(N)^{-1}$.  For instance, if $C = O(1)$ then one has
$$
c(\delta, C) \gg_{\vc} \exp(-C_{s, \eps} \delta^{-\frac{2}{s-2-\eps}}),
$$
which implies that
\begin{equation}\label{ultimate dependence}
\delta \ll_{\vc, \eps} \brac{\log\omega(N)}^{-\frac{s}{2} + 1 + \eps}.
\end{equation}

In view of Lemma \ref{L2 suffices} and the discussion which precedes it, our aim in the remainder of this note is to provide sufficient conditions a majorant $\nu$ should satisfy to ensure that if $0 \leq f \leq \nu$ with $\sum_n f(n) \geq \delta N$ then there exists a function $g$ which is dense (as in \eqref{density assumption}), which has bounded $L^2$-norm (as in \eqref{L2 assumption}), and such that $\|\hat{f} - \hat{g}\|_\infty$ is small. 
As previously observed in \eqref{zero fourier coefficient}, non-negative functions which are close in the $L^\infty$-Fourier norm are also close in the $L^1$-norm, so that
$$
\sum_n g(n) = \sum_n f(n) + O\brac{\bignorm{\hat{f} - \hat{g}}_\infty}.
$$
Hence the density of $f$ automatically implies the density of $g$.  We may therefore drop the requirement that our approximant $g$ is dense, as this follows from the Fourier approximation.  Our aim is therefore to answer the following question.

\begin{question*}
What conditions does a majorant $\nu$ on $[N]$ need to satisfy in order to ensure that any function $0 \leq f\leq \nu$ has a non-negative approximant $g$ with bounded $L^2$-norm and such that the difference $\|\hat{f}- \hat{g}\|_\infty$ is small? 
\end{question*}

Any result which provides conditions answering this question we call a \emph{bounded approximation lemma}, since we are attempting to approximate our undbounded function $f$ by a function $g$ which exhibits less growth, as measured by the $L^2$-norm.

\subsection{Notation}

In order to be consistent with the normalisation of our Fourier transform \eqref{fourier transform}, we define the $L^p$-norm of a function on the integers $f: \Z \to \C$ with respect to counting measure, so that
$$
\norm{f}_p := \brac{\sum_n |f(n)|^p}^{1/p}.
$$
For functions on $\T$, all $L^p$-norms are taken with respect to the Haar probability measure, so that for finitely supported $f : \Z \to \C$ we have
$$
\bignorm{\hat{f}}_2 = \norm{f}_2
$$
Notice that if $\nu$ is a majorant then we also have the identity
\begin{equation}\label{majorant1infinity}
\norm{\hat{\nu}}_\infty = \norm{\nu}_1.
\end{equation}

\section{Green's $L^\infty$-bounded approximation lemma}\label{green approx sec}

In this section we give a proof of perhaps the simplest bounded approximation lemma, originating in Green \cite{greenprimes}.  Not only does this yield an approximant with bounded $L^2$-norm, but also bounded $L^\infty$-norm, so in some sense this approximant has the best possible boundedness properties.  The price to be paid for such good boundedness is the quality of our final Fourier approximation $\hat{f} \approx \hat{g}$.

\begin{definition}[Fourier decay] We say that a majorant $\nu$ on $[N]$ has \emph{Fourier decay of level $\theta$} if
$$
\norm{\hat{\nu} - \hat{1}_{[N]}}_\infty \leq \theta N.
$$
\end{definition} 
Notice from \eqref{majorant1infinity}, that if a majorant has Fourier decay of level $\theta$ then
$$
\norm{\nu}_1 = N + O(\theta N).
$$

\begin{definition}[Restriction at $p$] We say that a majorant $\nu$ supported on $[N]$ satisfies a \emph{restriction estimate at exponent $p$} if 
$$
\sup_{|\phi| \leq \nu} \int_{\T} \abs{\hat{\phi}(\alpha)}^p \intd \alpha \ll_p \norm{\nu}_1^p N^{-1}.
$$
\end{definition}


\begin{theorem}[Green \cite{greenprimes}]\label{Green1}
Suppose that the majorant $\nu$ has Fourier decay of level $\theta$ and satisfies a restriction estimate at exponent $p$.  Then for any $0 \leq f \leq \nu$ there exists $0 \leq g \ll 1_{[N]}$ such that
$$
\bignorm{\hat{f} - \hat{g}}_\infty \ll_p \log(\theta^{-1})^{-\recip{p+2}} N.
$$
\end{theorem}

As the function $g$ delivered by this theorem is genuinely bounded, we call this an \emph{$L^\infty$-bounded approximation lemma}.

We begin the proof of Theorem \ref{Green1} by defining the large spectrum of $f$ to be the set 
$$
\Spec(f, \eta \norm{\nu}_1) := \set{\alpha \in \T : |\hat{f}(\alpha)| \geq \eta \norm{\nu}_1}.
$$
Define the Bohr set with frequency set $S := \Spec(f, \eta \norm{\nu}_1)$ and width $\eps \leq 1/2$ by
$$
B(S, \eps) := \set{n \in [-\eps N, \eps N] : \norm{n\alpha} \leq \eps\quad( \forall \alpha \in S)}.
$$
Write $\sigma$ for the normalised characteristic function of $B:= B(S, \eps)$, so that
$$
\sigma := |B|^{-1} 1_{B}.
$$
Then we define
\begin{equation}\label{g defn}
g := f * \sigma * \sigma,
\end{equation}
where, for finitely supported $f_i$, we set
$$
f_1 * f_2(n) := \sum_{m_1+m_2 = n} f_1(m_1) f_2(m_2).
$$

We first estimate $|\hat{f}- \hat{g}|$.  The key identity is 
$$
\widehat{f_1*f_2} = \hat{f}_1\hat{f}_2.
$$
If $\alpha \notin \Spec(f, \eta N)$ then we have
$$
|\hat{f}(\alpha)- \hat{g}(\alpha)| = |\hat{f}(\alpha)||1 - \hat{\sigma}(\alpha)^2| \leq 2 \eta \norm{\nu}_1 \ll \eta N.
$$
If $\alpha \in \Spec(f, \eta N)$, then for each $n \in B$ we have
$
e(\alpha n) = 1 + O(\eps).
$  
Hence
$
\hat{\sigma}(\alpha) = 1 + O(\eps),
$
and consequently
$$
|\hat{f}(\alpha)- \hat{g}(\alpha)| = |\hat{f}(\alpha)||1+\hat{\sigma}(\alpha)||1 - \hat{\sigma}(\alpha)| \ll \norm{\nu}_1\eps  \ll \eps N.
$$
Combining both cases gives 
\begin{equation}\label{fourierestimate}
\bignorm{\hat{f} - \hat{g}}_\infty \ll (\eps + \eta) N.
\end{equation}

It remains to show that $g$ is bounded.  By positivity and orthogonality, we have 
\begin{align*}
g(n)  = \sum_{x+y+z = n} f(x) \sigma(y) \sigma(z) &\leq \sum_{x+y+z = n} \nu(x) \sigma(y) \sigma(z)\\ & = \int_\T \hat{\nu}(\alpha) \hat{\sigma}(\alpha)^2 e(-\alpha n) \intd \alpha. 
\end{align*}
It therefore suffices to show that
\begin{equation}\label{requirement1}
\int_\T \hat{\nu}(\alpha) \hat{1}_B(\alpha)^2 e(-\alpha n) \intd \alpha \ll |B|^2.
\end{equation}

Inserting our Fourier decay assumption and using Parseval, we have
\begin{align*}
\int_\T \hat{\nu}(\alpha) \hat{1}_B(\alpha)^2 e(-\alpha n) \intd \alpha & \leq \int_\T \hat{1}_{[N]}(\alpha) \hat{1}_B(\alpha)^2 e(-\alpha n) \intd \alpha + \theta N \int_\T| \hat{1}_B(\alpha)|^2 \intd\alpha\\
& = \sum_{x+y+z = n} 1_{[N]}(x) 1_B(y) 1_B(z) + \theta N |B| \\
& \leq |B|^2 + \theta N |B|.
\end{align*}
We therefore obtain \eqref{requirement1} provided that
\begin{equation}\label{requirement2}
\theta N \ll |B|.
\end{equation}

By Lemma \ref{bohr set lower bound} we have
$
|B| \geq \eps^{O_p(\eta^{-p-1})} N,
$
so \eqref{requirement2} follows provided that
$
\theta \leq \eps^{C_p \eta^{-1-p}}.
$
In view of \eqref{fourierestimate}, let us take $\eps = \eta$ with 
$
\theta = \eps^{C_p \eta^{-1-p}}.
$
Then
$$
\log(\theta^{-1}) \leq C_p \log(\eps^{-1}) \eps^{-1-p} \ll_p \eps^{-2-p}.
$$
This implies that
$$
\bignorm{\hat{f} - \hat{g}}_\infty \ll \eps N \ll_p \brac{\log(\theta^{-1})}^{-\recip{2+p}} N,
$$
which completes the proof of Theorem \ref{Green1}.

\section{Helfgott and De Roton's $L^2$-bounded approximation lemma}\label{HDR approx sec}

For quantitative applications, a drawback of Green's bounded approximation lemma is the dependence of the final Fourier bound $\|\hat{f}-\hat{g}\|_\infty$ on the level of Fourier decay $\theta$ exhibited by the majorant $\nu$.  Typically our majorant satisfies a Fourier decay assumption of the form
$$
\|\hat{\nu} - \hat{1}_{[N]}\|_\infty \ll N(\log N)^{-c}.
$$
This results in a final Fourier bound of the form
\begin{equation}\label{loglog fourier bound}
\|\hat{f} - \hat{g}\|_\infty \ll N(\log \log N)^{-\recip{p+2}}.
\end{equation}
Notice that this loses a logarithm over our assumed Fourier bound, even when $f = \nu$, where we may take $g = 1_{[N]}$.

In the process of improving Green's bound \cite{greenprimes} for Roth's theorem in the primes, Helfgott and De Roton \cite{helfgottderoton} developed a new variant of the bounded approximation lemma which removes this logarithmic loss from the final Fourier bound.  There is a price to be paid for this improvement.  The first is that the approximant may no longer be an $L^\infty$-bounded function, but instead has the weaker property of being bounded in the $L^2$-norm.  However, as Helfgott and De Roton observed in Lemma \ref{L2 suffices}, this is not really an impediment.  A more serious price must be paid in making a stronger assumption on their majorant $\nu$ than Fourier decay.

\begin{definition}[Two-point correlation estimate]  
Let us say that a majorant satisfies a \emph{two point correlation estimate} if for any non-zero $m$ we have
\begin{equation}\label{two point correlation}
\sum_n \nu(n) \nu(n + m) \ll N.
\end{equation}
\end{definition}

\begin{definition}[$L^2$-boundedness of level $\theta$]
We say that a majorant $\nu$ on $[N]$ has \emph{$L^2$-boundedness of level $\theta$} if 
\begin{equation}\label{L2 level}
\sum_n \nu(n)^2 \leq \theta N^2.
\end{equation}
\end{definition}

Notice that if a majorant satisfies the $L^\infty$-bound $\nu \leq \theta N$, then the $L^1$-assumption \eqref{L1 assumption} gives $L^2$-boundedness of level $\theta$.  

\begin{theorem}[Helfgott and De Roton \cite{helfgottderoton}]\label{helfgottderoton}
Suppose that the majorant $\nu$ satisfies a restriction estimate at exponent $p$, a two-point correlation estimate and has $L^2$-boundedness of level $\theta$. 
Then for any $0 \leq f \leq \nu$ there exists $g \geq 0$ such that $\sum_n g(n)^2 \ll N$ and 
\begin{equation*}\label{HDR fourier approx}
\bignorm{\hat{f} - \hat{g}}_\infty \ll_p \log(\theta^{-1})^{-\recip{p+2}} N.
\end{equation*}
\end{theorem}

In applications the $\theta$ parameter resulting from the level of $L^2$-boundedness \eqref{L2 level} is of the form $N^{-c}$ for some absolute constant $c > 0$.  In practice, this is much smaller than the Fourier decay parameter $\theta$ that one might hope to obtain for $\nu$, which is usually of the form $(\log N)^{-c}$.  This results in a final Fourier approximation of the form
\begin{equation*}
\|\hat{f} - \hat{g}\|_\infty \ll N(\log N)^{-\recip{p+2}},
\end{equation*}
which saves a logarithm over the estimate given in \eqref{loglog fourier bound}.

The proof of Theorem \ref{helfgottderoton} is similar to that given in \S \ref{green approx sec}.  Adopting the notation of \S \ref{green approx sec}, we define $g$ as in \eqref{g defn}, albeit with one less convolution
$$
g := f * \sigma.
$$
The same argument given in \S \ref{green approx sec} gives the Fourier bound
$$
\|\hat{f} - \hat{g}\|_\infty \ll (\eps + \eta) N, 
$$
so we take $\eps = \eta$ to yield $\|\hat{f} - \hat{g}\|_\infty \ll \eps N$.

Now our treatment departs from that given previously as we are aiming to prove the $L^2$-bound $\sum_n g(n)^2 \ll N$, which is equivalent to
\begin{equation*}\label{hoped for L2 bound}
\sum_{n_1 - n_2 = m_1 - m_2} f(n_1)f(n_2)1_B(m_1)1_B(m_2) \ll N |B|^2.
\end{equation*}
Utilising $f \leq \nu$, this equals
$$
\sum_m \brac{\sum_{m_1 - m_2 = m} 1_B(m_1)1_B(m_2)} \brac{ \sum_n \nu(n) \nu(n+m)}.
$$
Incorporating our assumptions \eqref{two point correlation} and \eqref{L2 level}, this is at most
$$
\theta N^2 |B| + O(N|B|^2).
$$

We have therefore obtained $L^2$-boundedness provided that $\theta N \ll |B|$.  Recalling Lemma \ref{bohr set lower bound}, it suffices to have
$$
\theta \leq \eps^{O_p(\eps^{-p-1})},
$$
or equivalently 
$$
C_p \log(\theta^{-1})^{-\recip{p+2}} \leq \eps.
$$
Taking the smallest permissible value of $\eps$ then yields Theorem \ref{helfgottderoton}.

\section{Naslund's $L^k$-bounded approximation lemma}\label{naslund section}

As is apparent in the deduction of the density bound \eqref{ultimate dependence}, if one is interested in quantitative bounds for sets lacking solutions to \eqref{linear equation}, then the quantitative dependence in \eqref{c dependence} is important.  Ideally, one would hope not to lose too much by passing from the constant $c(\delta)$ available for the characteristic function of a dense set, to the constant $c(\delta, C)$ available for a function with bounded $L^2$-norm.  In a perfect world, this loss would take the form, say
$$
c(\delta, C) = c\bigbrac{\tfrac{\delta}{100C}},
$$
whereas the proof of Lemma \ref{L2 suffices} yields
$$
c(\delta, C) = (\delta/2)^s c\bigbrac{\tfrac{\delta^2}{4C}}.
$$
The occurrence of the factor $\delta^s$ in \eqref{c dependence} seems unavoidable. Fortunately, this factor is not too costly, since it is much larger than the lower bound \eqref{bloom dependence} for $c(\delta)$.  A more significant loss is the appearance of $\delta^2$ within the function $c\bigbrac{\tfrac{\delta^2}{4C}}$, which ultimately stems from the lower bound \eqref{large value density}.

As observed by Naslund \cite{naslund}, one may replace the use of Cauchy--Schwarz in \eqref{cauchy employment} by H\"older's inequality in order to replace the occurrence of $\delta^2$ by, essentially, $\delta^{1+\eps}$.  This improvement ultimately stems from aiming for an $L^k$-bounded approximant for some large $k$ (depending on $\eps$), rather than the weaker $L^2$-approximant of Helfgott and De Roton. Since the $L^k$-norm of a finitely supported function tends to the $L^\infty$-norm with $k$, one may think of $L^k$-boundedness as a half-way house between the weak notion of $L^2$-boundedness and the strong $L^\infty$-notion.  

\begin{lemma}[$L^k$-boundedness suffices]\label{Lp suffices}  Let $c_1 + \dots + c_s = 0$.  Then for any $\delta > 0$, any constant $C$ and any $k \geq 2$ there exists $c(\delta, C, k) > 0$ such that the following holds.
Suppose that $g:\Z \to [0, \infty)$ is a non-negative function supported on $[N]$ which has bounded $L^k$-norm
\begin{equation*}\label{Lp assumption}
\sum_n g(n)^k \leq C N.
\end{equation*}
Then the density assumption
$
\sum_n g(n) \geq \delta N
$
implies that
$$
\sum_{\vc\cdot \vx = 0} g(x_1) \dotsm g(x_s) \geq c(\delta, C, k) N^{s-1},
$$
Moreover, one may take 
\begin{equation}\label{cp dependence}
c(\delta, C, k) = (\delta/2)^s c\brac{\brac{\tfrac{\delta}{2C}}^{1 + \recip{k-1}}},
\end{equation}
where $c(\delta)$ is the constant appearing in Roth's theorem \eqref{lower bound}.
\end{lemma}

\begin{proof}
We proceed as in Helfgott and De Roton's argument for Lemma \ref{L2 suffices}, albeit using H\"older's inequality to give the upper bound
$$
\sum_{x \in B} g(x) \leq |B|^{1-\recip{k}}\brac{\sum_x g(x)^k}^{\recip{k}}.
$$
This results in the lower bound
$$
|B| \geq \brac{\frac{\delta}{2C}}^{1 + \recip{k-1}} N,
$$
from which \eqref{cp dependence} follows.
\end{proof}

The price to paid for obtaining an approximant with the stronger notion of $L^k$-boundedness is that one's majorant must now satisfy a more stringent correlation condition.

\begin{definition}[$k$-point correlation estimates]  
Let us say that a majorant satisfies the \emph{$k$-point correlation estimates} if for any distinct $m_1, \dots, m_l$ with $l \leq k$ we have
\begin{equation}\label{many point correlation}
\sum_n \nu(n+m_1) \dotsm \nu(n+m_l) \ll N.
\end{equation}
\end{definition}

\begin{definition}[$L^\infty$-boundedness of level $\theta$]
We say that a majorant $\nu$ on $[N]$ has \emph{$L^\infty$-boundedness of level $\theta$} if for all $n$ we have
\begin{equation}\label{L2 level}
\nu(n) \leq \theta N.
\end{equation}
\end{definition}

By assumption a majorant satisfies $\sum_n \nu(n) = (1+o(1))N$, so that the level of $L^\infty$-boundedness is at worst $O(1)$, and unless $\nu$ is concentrated on a bounded set, will be $o(1)$ in applications.

\begin{theorem}[Naslund \cite{naslund}]\label{naslund}
Suppose that $\nu$ is a majorant on $[N]$ satisfying a restriction estimate at exponent $p$, with $L^\infty$-boundedness of level $\theta$ and satisfying the $k$-point correlation estimates with 
\begin{equation}\label{k limitation}
k \leq \trecip{2}\sqrt{\log(\theta^{-1})}.
\end{equation} 
Then for any $0 \leq f \leq \nu$ there exists $g \geq 0$ such that $\sum_n g(n)^k \ll N$ and 
\begin{equation*}
\bignorm{\hat{f} - \hat{g}}_\infty \ll_p \log(\theta^{-1})^{-\recip{p+2}} N.
\end{equation*}
\end{theorem}

In order to employ this result in conjunction with Lemma \ref{Lp suffices}, one might hope, in view of \eqref{cp dependence}, to take $k = \ceil{1+\eps^{-1}}$.
As mentioned previously, in applications we expect to be able to obtain $L^\infty$-boundedness of level $N^{-c}$.  Hence \eqref{k limitation} certainly follows if $N \geq \exp(C\eps^{-2})$.  Provided that one can prove the $(1+\eps^{-1})$-point correlation estimates, one may then deduce a lower bound in \eqref{cp dependence} of the form 
$$
c(\delta, C) \gg_\eps (\delta/2)^s c\brac{\brac{\tfrac{\delta}{2C}}^{1 + \eps}}.
$$

\begin{proof}
The construction is the same as in the proof of Theorem \ref{helfgottderoton}.  Just as in that proof we take $\eta = \eps$ to obtain an approximant $g \geq 0$ with $\|\hat{f} - \hat{g}\|_\infty \ll \eps N$.  Our task then reduces to determining a permissible value of $\eps$ which allows one to show that
$$
\sum_n \brac{\sum_m \nu(n-m)1_B(m)}^k \ll N|B|^k.
$$
Expanding out the $k$th power and noting that $B = -B$, this is equivalent to the estimate
\begin{equation}\label{k point to prove}
\sum_{m_1, \dots, m_k \in B}\ \sum_n \nu(n+m_1)\dotsm \nu(n+m_k) \ll N|B|^k.
\end{equation}

Fix a choice of $(m_1, \dots, m_k) \in B^k$ and let $(m_1', \dots, m_l')$ denote the distinct values occurring in this choice, written in the order in which they appear in the tuple, and with respective multiplicities $k_1, \dots, k_l$.
Then by the level of $L^\infty$-boundedness and the $k$-point correlation estimate, we have
 \begin{align*}
 \sum_n \nu(n+m_1)\dotsm \nu(n+m_k) & = \sum_n \nu(n+m_1')^{k_1}\dotsm \nu(n+m_l')^{k_l} \\
 & \leq (\theta N)^{k - l} \sum_n \nu(n+m_1') \dotsm \nu(n+m_l')\\
 & \ll (\theta N)^{k - l }N.
 \end{align*}
 
 By choosing one of the symbols `$=$' or `$\neq$' for each pair of indices $1 \leq i < j \leq k$, we see that for each choice of tuple $(m_1',\dots, m_l') \in B^l$ with distinct entries, there are at most $2^{\binom{k}{2}}$ choices of $(m_1,\dots, m_k) \in B^k$ giving rise to $(m_1', \dots, m_l')$.  It follows that
 \begin{align*}
 \sum_{m_1, \dots, m_k \in B}\ \sum_n \nu(n+m_1)\dotsm \nu(n+m_k) & \ll \sum_{l =1}^k2^{\binom{k}{2}} |B|^l  (\theta N)^{k-l} N\\
 & \leq |B|^k N \max_{1 \leq l \leq k} \brac{\frac{k2^{\binom{k}{2}} \theta N}{|B|}}^l
 \end{align*}
 The required bound \eqref{k point to prove} then follows on ensuring that $|B| \geq k2^{\binom{k}{2}} \theta N$, which from Lemma \ref{bohr set lower bound}, follows if 
 $$
 \eps^{O_p(\eps^{-p-1})} \geq k2^{\binom{k}{2}} \theta.
 $$
 This in turn follows if 
 $$
 \log(\theta^{-1}) \geq \log k + \binom{k}{2}\log 2 + C_p\eps^{-p-2}.
 $$
By \eqref{k limitation} and the inequality $\log k + \binom{k}{2} \log 2\leq k^2$, it suffices to take
 $$
 \eps = \brac{\frac{2C_p}{ \log(\theta^{-1})}}^{\recip{p+2}}.
 $$
 \end{proof}

\section{The Hahn--Banach approach}\label{BoundedApproxSec}

The fact that a majorant satisfies a restriction estimate at some exponent $p$ is essential in applications of the transference principle to the circle method, see for instance \cite{greentaorestriction, harper, densesquares}.  In general, if a function $f$ is efficiently bounded by a majorant $\nu$, one can count solutions to a linear equation in $s$ variables weighted by $f$ provided that one can obtain a restriction estimate for $\nu$ for some $p < s$. 

The limited use of the restriction esimate in the proof of theorems \ref{Green1}, \ref{helfgottderoton} and \ref{naslund} suggests that it may not be necessary for a majorant $\nu$ to satisfy such an estimate in order for $f \leq \nu$ to have a bounded approximation.  This was first shown by Gowers \cite{gowers10} and, independently, by Reingold et al \cite{RTTVnote}.  It turns out that removing the quantitative dependence of the final Fourier approximation on the restriction parameter gives a marginally stronger bound.  
 Their method extends to give a bounded approximation lemma for norms other that the $L^\infty$-Fourier norm, giving an alternative derivation of the transference principle found in \cite{greentaoprimes, taoziegler}, and which is essential for applications to systems of linear equations such as \cite{greentaolinear, matthiesen}.  In this section we give an exposition of their argument limited to the simpler Fourier-analytic context.

In common with Green's transference principle, the approximation theorem assumes some level of Fourier decay.  Although quantitativley weaker than the assumption of a correlation condition, this is in some sense a more useful assumption for applications, such as \cite{densesquares}, where the correlation estimates \eqref{two point correlation} and \eqref{many point correlation} do not necessarily hold.

\begin{theorem}[\cite{gowers10, RTTVnote}]\label{lem:BoundedApprox}
Suppose that the majorant $\nu$ has Fourier decay of level $\theta$.  Then for any $0 \leq f \leq \nu$ there exists a bounded function $0 \leq g \leq 1_{[N]}$ such that
$$
\bignorm{\hat{f} - \hat{g}}_{\infty} \ll \log(1/\theta)^{-3/2}N.
$$
\end{theorem}

Both \cite{gowers10} and \cite{RTTVnote} follow similar lines in proving this result, employing either the supporting hyperplane theorem or the minimax theorem to give the existence of $g$, rather than the explicit construction of \S\S \ref{green approx sec}--\ref{naslund section}.  Both of these subsidiary results are closely related to the finite dimensional Hahn--Banach theorem.  We give a complete account of the necessary background in the appendices.  

We identify the set of functions $f : \Z \to \C$ whose support is contained in $[N]$ with the finite dimensional space $\C^N$.  Then the functional
$$
\norm{f} := \bignorm{\hat{f}}_\infty
$$
forms a norm on this space.  Recall that we define the dual norm by
$$
\norm{\phi}^*:=\sup_{\norm{f} \leq 1} \abs{\ang{f,\phi}},
$$
where
$$
\ang{f,\phi} := \sum_{n=1}^N f(n)\overline{\phi(n)}.
$$
One can check that this is itself a norm on $\C^N$, and it follows directly from the definition that for any $f, \phi \in \C^N$ we have the inequality
\begin{equation}\label{eqn:DualIneq}
\abs{\ang{f,\phi}} \leq \norm{f} \norm{\phi}^*.
\end{equation}

\begin{lemma}[Properties of the dual of $\bignorm{\hat{f}}_\infty$]\label{DualProps} 
{\ }
\begin{enumerate}[(i)]
\item\label{item:algebra} (Algebra property)  
\begin{equation}\label{eqn:algebra}
\norm{\phi_1\phi_2}^* \leq \norm{\phi_1}^*\norm{\phi_2}^*.
\end{equation}
\item\label{item:infty}  ($L^\infty$--compatibility) 
$$
\norm{\phi}_{\infty} \leq \norm{\phi}^*.
$$
\item\label{item:reality}  (Real compatibility)  
$$
\norm{\Realnobrac\,\phi}^* \leq \norm{\phi}^*.
$$
\item\label{item:duality} (Duality)  For any $f \in \C^N$ there exists $\phi \in \C^N$ with $\norm{\phi}^* = 1$ such that 
$$
\norm{f} = \Realnobrac\ang{f ,  \phi}.
$$
\end{enumerate}
\end{lemma}

\begin{proof}  
Let $f \in \C^N$ with $\norm{f} \leq 1$.  Then by \eqref{eqn:DualIneq} we have
\begin{align*}
\abs{\ang{f, \phi_1\phi_2}}  = \abs{\ang{f \overline{\phi}_1, \phi_2}} \leq \bignorm{f\overline{\phi}_1}\bignorm{\phi_2}^*.
\end{align*}
For $\alpha \in \T$ write $e_\alpha(n) := e(\alpha n)$.  Then
\begin{align*}
\bigabs{ \widehat{f\overline{\phi}_1} (\alpha)} & = \abs{\ang{fe_\alpha, \phi_1}}\\
& \leq \norm{fe_\alpha} \norm{\phi_1}^*.
\end{align*}
By a change of variables we have $\norm{fe_{\alpha}} = \norm{f} \leq 1$.  Thus $\norm{f\overline{\phi}_1} \leq \norm{\phi_1}^*$, which establishes \eqref{item:algebra}.

To prove \eqref{item:infty} it suffices, by homogeneity, to show that the ball
$
B^*:=\set{\phi \in \C^N:\norm{\phi}^* \leq 1}
$ is contained in the ball
$
B_{\infty}:= \set{\phi \in \C^N:\norm{\phi}_{\infty} \leq 1}.
$
By equivalence of norms on finite dimensional spaces, $B^*$ is a bounded subset of $\C^N$.  Suppose that $\phi$ is an element of $B^* \setminus B_{\infty}$, so that $|\phi(n)| > 1$ for some $n \in [N]$.  By the algebra property \eqref{eqn:algebra}, $\phi^k \in B^*$ for all $k \in \N$.  Yet $|\phi^k(n)|\to \infty$ as $k \to \infty$, contradicting boundedness.

For \eqref{item:reality}, we first note that $\norm{\cdot}$ is invariant under complex conjugation, since
\begin{align*}
\norm{\overline{f}} = \sup_{\alpha}\abs{\sum_n \overline{f(n)} e(\alpha n)} 
= \sup_{\alpha}\abs{\sum_n f(n) e(-\alpha n)} = \norm{f}.
\end{align*}
It follows that $$
\norm{\overline{\phi}}^*=\sup_{\norm{f} \leq 1} \abs{\ang{f,\overline{\phi}}} = \sup_{\norm{f} \leq 1} \abs{\ang{\overline{f},\phi}} = \sup_{\norm{g} \leq 1} \abs{\ang{g,\phi}} = \norm{\phi}^*.
$$
Hence by the triangle inequality and homogeneity
$$
\norm{\Realnobrac\,\phi}^* = \norm{\frac{\phi + \overline{\phi}}{2}}^* \leq \frac{\norm{\phi}^* + \norm{\phi}^*}{2} = \norm{\phi}^*.
$$

 To prove \eqref{item:duality} it suffices to prove that for $f \neq 0$ there exists $\phi \neq 0$ such that $\Realnobrac\ang{f, \phi}\geq \norm{f} \norm{\phi}^*$, as the reverse inequality follows from \eqref{eqn:DualIneq}, and \eqref{item:duality} then follows by homogeneity.  Consider the convex set $C = \set{g : \norm{g} \leq  \norm{f}}$.    Since $f \notin \interior(C)$, the complex supporting hyperplane theorem (Corollary \ref{lem:CompSuppHyp}) gives the existence of $\phi \neq 0$ such that for any $\norm{g} \leq  \norm{f}$ we have 
$$
\Realnobrac{\ang{f , \phi}} \geq \Realnobrac{\ang{g ,\phi}}.
$$
For each $g$ with $\norm{g} \leq \norm{f}$ there exists $|\theta|=1$ such that $$|\ang{g, \phi}| =  \theta\ang{ g, \phi} = \ang{\theta g, \phi}=   \Realnobrac{\ang{\theta g, \phi}}.$$  Notice that $\norm{\theta g} \leq \norm{f}$ also, therefore
$$
 \Realnobrac{\ang{f , \phi}} \geq \Realnobrac{\ang{\theta g ,\phi}} = |\ang{g , \phi}|.
$$
Hence by homogeneity
$$
\Realnobrac\ang{f , \phi} \geq \sup_{\norm{g} \leq \norm{f}} |\ang{g , \phi}| = \norm{f} \sup_{\norm{g} \leq 1} |\ang{g , \phi}| = \norm{f} \norm{\phi}^*.
$$
\end{proof}

\begin{proof}[Proof of Theorem \ref{lem:BoundedApprox}]  

We prove the contrapositive, supposing there exists $0 \leq f \leq \nu$ such that for any $0 \leq g \leq 1_{[N]}$ we have
$$
\bignorm{f - {g}} > \eps N.
$$ 
Our aim is to deduce that $\norm{\nu - 1_{[N]}}> \exp\brac{-C\eps^{-2/3}}  N$.  If $\norm{\nu - 1_{[N]}} >   N$ we are done, so we may assume that $\norm{\nu - 1_{[N]}} \leq   N$. In particular, it is useful to note for later that
\begin{equation}\label{eqn:BoundedL1}
\norm{\nu}_{1} = \norm{\hat{\nu}}_{\infty} \leq \norm{ \hat{1}_{[N]}}_{\infty} + \norm{\hat{\nu} - \hat{1}_{[N]}}_{\infty} \leq 2N.
\end{equation}

By Lemma \ref{DualProps} \eqref{item:duality}, for each $0 \leq g \leq 1_{[N]}$ there exists $\phi_{g}$ with $\norm{\phi_{g}}^* = 1$ such that 
\begin{equation}\label{DualAssumption}
\Realnobrac\ang{f - g , \phi_{g}} >  \eps N.
\end{equation}
Consider the subsets of $\C^N$ given by 
$$
A := \set{ g-f : 0 \leq g \leq 1_{[N]}} \quad \text{and}\quad B:= \set{\phi : \norm{\phi}^* \leq 1}.
$$
One can check that both $A$ and $B$ are convex, compact and non-empty.  Moreover, $A$ is the convex hull of the finite set $ \set{1_{S} -f :S \subset [N]}$. Applying the minimax theorem (Corollary \ref{CompMini}), there exists $0 \leq {g_0} \leq 1_{[N]}$ and $\norm{{\phi_0} }^* \leq 1$ such that for any $0 \leq g \leq 1_{[N]}$ and $\norm{\phi }^* \leq 1$ we have
$$
\Realnobrac\ang{{g_0} -f,\phi} \geq \Realnobrac\ang{g-f ,{\phi_0}}.
$$
In particular, using \eqref{DualAssumption} we see that for any $0 \leq g \leq 1_{[N]}$ we have
$$
\Realnobrac\ang{f -g,{\phi_0}} \geq \Realnobrac\ang{f - {g_0},\phi_{{g_0}}}> \eps N.
$$

Set $\psi:= \Realnobrac\, {\phi_0}$ and write $\psi_+$ for the positive part of $\psi$.  Taking $g := 1_{\psi  \geq 0}$, non-negativity gives that
$$
\ang{\nu , \psi_{+} }\geq \ang{f , \psi_+ }\geq \ang{f ,\psi} = \Realnobrac\ang{f , {\phi_0}} > \Realnobrac\ang{g , {\phi_0}} +  \eps N =\ang{1_{[N]} ,\psi_+} + \eps N.
$$
Therefore
$$
 \ang{\nu - 1_{[N]} , \psi_+ } >\eps N.
$$

By $L^\infty$--compatibility (Lemma \ref{DualProps} \eqref{item:infty}) we have
$$
\norm{\psi}_\infty \leq \norm{\phi_0}_\infty \leq \norm{\phi_0}^* \leq 1.
$$
Hence by the Weierstrass polynomial approximation theorem (Lemma \ref{Weier}) there exists a polynomial $P$ of degree at most $C \eps^{-2/3}$ and height at most $\exp(C \eps^{-2/3})$ such that 
$$
\norm{P \circ \psi - \psi_+}_{\infty} \leq \trecip{4} \eps 
$$
Using this and the observation \eqref{eqn:BoundedL1}, we see that
\begin{align*}
\ang{\nu - 1_{[N]} , P\circ \psi} & = \ang{\nu - 1_{[N]}, \psi_+} + \ang{\nu - 1_{[N]} ,P\circ \psi - \psi_+}\\
& \geq  \eps N - \norm{\nu - 1_{[N]}}_{1} \norm{P\circ \psi - \psi_+}_{\infty}\\
& \geq \trecip{2}\eps N.
\end{align*}
By \eqref{eqn:DualIneq} it follows that
\begin{equation}\label{eqn:AllBarAlg}
\norm{\nu - 1_{[N]}} \norm{P\circ\psi}^*  \geq \trecip{2} \eps N.
\end{equation}

By real compatibility (Lemma \ref{DualProps} \eqref{item:reality}), we have
$
\norm{\psi}^* \leq \norm{\phi_0}^* \leq 1.
$
Hence by the algebra property (Lemma \ref{DualProps} \eqref{item:algebra}) and the triangle inequality, we deduce that 
$$
\norm{P\circ \psi}^* \ll \exp(C\eps^{-2/3}).
$$ 
Combining this with \eqref{eqn:AllBarAlg} finally yields the required bound.
\end{proof}

\appendix

\section{The large spectrum and Bohr sets}

As in \S \ref{green approx sec} we define the $(\eta \norm{\nu}_1)$-large spectrum of $f$ to be the set 
$$
\Spec(f, \eta \norm{\nu}_1) := \set{\alpha \in \T : |\hat{f}(\alpha)| \geq \eta \norm{\nu}_1}.
$$
Notice that this set is empty unless $\eta \leq 1$, which we assume throughout what follows.

\begin{lemma}\label{restrictionmeasure} Suppose that $\nu$ is a majorant on $[N]$ satisfying a restriction estimate at exponent $p$.  Then for any $0 \leq f \leq \nu$ we have
$$
\meas\bigbrac{\Spec(f, \eta \norm{\nu}_1)} \ll_p \eta^{-p}N^{-1}.
$$
\end{lemma}

\begin{proof}  We have
\begin{align*}
\meas\brac{\Spec(f, \eta N)} & \leq (\eta \norm{\nu}_1)^{-p} \int_{\Spec(f, \eta N)} |f(\alpha)|^p\intd\alpha\\
& \leq (\eta \norm{\nu}_1)^{-p} \int_{\T} |f(\alpha)|^p\intd\alpha.
\end{align*}
By the restriction estimate  we have
$$
\int_{\T} |f(\alpha)|^p\intd\alpha \ll_p \norm{\nu}_1^{p}N^{-1}.
$$
\end{proof}

Define the Bohr set with frequency set $S \subset \T$ and width $\eps \leq 1/2$ by
$$
B(S, \eps) := \set{n \in [-\eps N, \eps N] : \norm{n\alpha} \leq \eps\quad( \forall \alpha \in S)}.
$$

\begin{lemma}\label{bohr set lower bound}  Suppose that $\nu$ is a majorant on $[N]$ satisfying a restriction estimate at exponent $p$.  Then for $0 \leq f \leq \nu$ and $S = \Spec(f, \eta \norm{\nu}_1)$ we have
$$
|B(S, \eps)| \geq \eps^{O_p(\eta^{-p-1})} N.
$$
\end{lemma}

\begin{proof} 
Set
$$
M := \ceil{4\pi N \eta^{-1}}
$$ and partition $\T$ into $M$ half-open intervals of length $M^{-1}$.  Let $I_1, \dots, I_r$ denote those intervals which intersect $S = \Spec(f, \eta N)$.  We claim that
\begin{equation*}\label{unioncontainment}
\bigcup_{i=1}^r I_i \subset \Spec(f,\trecip{2} \eta \norm{\nu}_1).
\end{equation*}
To see this, let us fix a choice of $\alpha_i \in I_i\cap S$ for each $i$.  If $\alpha \in I_i$ then $\norm{\alpha - \alpha_i} \leq \eta/(4\pi N)$ so that
\begin{align*}
|\hat{f}(\alpha)| & \geq |\hat{f}(\alpha_i)| - |\hat{f}(\alpha) - \hat{f}(\alpha_i)|\\
& \geq \eta \norm{\nu}_1 - \norm{f}_1 N2\pi \norm{\alpha - \alpha_i}\\
& \geq \trecip{2} \eta \norm{\nu}_1.
\end{align*}

By Lemma \ref{restrictionmeasure} we therefore have
$$
r \eta /N \ll \meas\bigbrac{\bigcup_{i=1}^r I_i} \ll_p \eta^{-p}/N,
$$
so that
$$
r \ll_p \eta^{-1-p}.
$$
One can check that
$$
B(\set{\alpha_1, \dots, \alpha_r}, \eps/2) \subset B(S, \eps).
$$
Therefore
$$
|B(S, \eps)| \geq |B(\set{\alpha_1, \dots, \alpha_r}, \eps/2)|.
$$

Set $T := \ceil{2/\eps}$ and partition $\T^r$ into $T^r$ half-open cubes of side-length $T^{-1}$.  By the pigeon-hole principle, some such cube $C$ contains the point $n(\alpha_1, \dots, \alpha_r)$ for at least $\recip{2}\eps N T^{-r}$ values of $n \in [0, \recip{2}\eps N]$.  Then $C - C \subset [ -T^{-1}, T^{-1}]^r$ contains at least $\recip{2}\eps N T^{-r}$ values of $n \in [- \recip{2}\eps N, \recip{2}\eps N]$.  In conclusion, we have shown that 
$$
|B(\set{\alpha_1, \dots, \alpha_r}, \eps/2)| \geq \trecip{2}\eps N \ceil{2/\eps}^{-r} \geq\ceil{2/\eps}^{-(r+1)} N.
$$
The lemma now follows.
\end{proof}

\section{The supporting hyperplane theorem}

In this appendix we give an account of the supporting hyperplane theorem, employed in \S \ref{BoundedApproxSec}, and also needed in the proof of the minimax theorem given in Appendix \ref{minimax section}.  The result is itself a weak version of the finite dimensional Hahn--Banach theorem and is standard.  However, we have not found a satisfactory reference for the version of the result we require.  

\begin{definition}[Affine independence]  We say $x_0, x_1, \dots, x_k \in \R^n$ are \emph{affinely dependent} if there exist $\lambda_i \in \R$ not all zero such that
$$
\sum_{i=1}^k \lambda_i x_i = 0 \quad \text{and}\quad \sum_{i=1}^k \lambda_i = 0.
$$
Equivalently, the differences $x_1 - x_0, \dots, x_k - x_0$ are linearly dependent.  
\end{definition}

\begin{lemma}\label{lem:OpenSimplex}
If $x_0, x_1, \dots, x_n \in \R^n$ are affinely independent, then the simplex
\begin{equation}\label{eqn:Simplex}
\set{\sum_{i=0}^n \lambda_i x_i : \lambda_i > 0, \ \sum_{i=0}^n \lambda_i = 1}
\end{equation}
is a non-empty open set.
\end{lemma}

\begin{proof}
The set 
$$
\Delta:=\set{ \mu \in \R^n : \mu_i > 0,\ \sum_{i=1}^n \mu_i < 1}
$$
is the finite intersection of $n+1$ open sets each containing $(1/2, \dots, 1/2)$, so is itself a non-empty open set.  

The simplex \eqref{eqn:Simplex} is equal to 
$$
\set{\sum_{i=1}^n \mu_i (x_i- x_0) : \mu_i > 0, \ \sum_{i=1}^n \mu_i < 1}, 
$$
which is the image of $\Delta$ under a map with continuous inverse.  Hence \eqref{eqn:Simplex} is open and non-empty. 
\end{proof}

Given $x \in \R^n$, write
$$
\abs{x}_\infty := \max_i |x_i| \quad \text{and} \quad B^\infty_\eps(x) := \set{y \in \R^n : \abs{x-y}_\infty < \eps}.
$$

\begin{lemma}\label{lem:AffPerturb}
For $x_0, x_1, \dots, x_n \in \R^n$  affinely independent, there exists $\eps > 0$ such that for any $y \in B^\infty_\eps(x_n)$ the vectors $x_0, \dots, x_{n-1}, y$ are also affinely independent.
\end{lemma}

\begin{proof}
Let $T$ denote the invertible linear map $\lambda \mapsto \sum_i \lambda_i (x_i -x_0)$.  Then there exists $C = C(x_i) > 0$ such that for any $v \in \R^n$ we have
$$
\abs{T^{-1} v}_\infty \leq C \abs{v}_\infty.
$$

Suppose that $x_0, \dots, x_{n-1}, x_n + v$ are affinely dependent.  Then there exist $\lambda_i$ with $\lambda_n = 1$ such that 
$$
 \sum_{i=1}^n \lambda_i (x_i - x_0) =  - v.
$$
Therefore $|T^{-1} v|_\infty \geq 1$, which in turn implies that $|v|_\infty \geq C^{-1}$.  The lemma now follows on taking $\eps = C^{-1}$.
\end{proof}

Given a subset $C$ of a topological space, write $\overline{C}$ for its closure and $\interior(C)$ for its interior.   

\begin{lemma}\label{lem:EmptyCase}
Let $C \subset \R^n$ be a convex set.  Then $\interior(C) = \emptyset$ if and only if $\interior(\overline{C}) = \emptyset$.
\end{lemma}

\begin{proof}  It suffices to prove the contrapositive of the `only if' direction.
Let $x \in \interior(\overline{C})$, so that there exists $\eps > 0$ such that 
$$
B_\eps^{\infty}(x) \subset \overline{C}.
$$
Taking $x_0 = x$ and $x_i = x + (\eps/2) e_i$, one sees that the set $B_\eps^{\infty}(x)$ contains $n+1$ affinely independent points.

By Lemma \ref{lem:AffPerturb} there exists $\delta > 0$ such that $B_\delta^{\infty}(x_n) \subset B_\eps^{\infty}(x)$ and for any $y \in B_\delta(x_n)$ the vectors $x_0, x_1, \dots, x_{n-1}, y$ are affinely independent.  Since $x_n \in \overline{C}$, there exists $x_n' \in B_\delta(x_n) \cap C$, so that $x_0, x_1, \dots, x_{n-1}, x_n'$ are affinely independent elements of $ B_\eps^{\infty}(x)$.

Repeating the above argument with $x_i$ in place of $x_n$, we see that we can find affinely independent $x_0', x_1' , \dots, x_n' \in \subset B_\eps^{\infty}(x) \cap C$.  It then follows from convexity and Lemma \ref{lem:OpenSimplex} that the set 
$$
\set{\sum_{i=0}^n \lambda_i x_i' : \lambda_i > 0, \ \sum_{i=0}^n \lambda_i = 1}
$$
is a non-empty open subset of $C$.  

\end{proof}

\begin{lemma}
Let $C$ be a convex subset of $\R^n$ with $x \in \interior(C)$ and $y \in \overline{C}$.  Then $\interior(C)$ contains the line segment
$$
[x, y) := \set{ (1-\lambda) x + \lambda y : \lambda \in [0, 1)}.
$$
\end{lemma}

\begin{proof}
Since $x \in \interior(C)$ there exists an open set $U \subset C$ with $x \in U$.  Let $z \in (x, y)$, so that there exists $\lambda \in (0,1)$ with
$$
z = (1-\lambda)x + \lambda y.
$$
Taking $\mu = \lambda^{-1}$ we have 
$$
y = \mu z + (1- \mu) x,
$$
so that $y$ is an element of the open set
$$
V:= \bigcup_{\mu > 1} \brac{ \mu z + (1-\mu) \cdot U}.
$$
Since $y \in \overline{C}$, there exists $y_1 \in V \cap C$.  Hence there exists $\mu_1 > 1$ and $u_1 \in U$ such that
$$
y_1 = \mu_1 z + (1- \mu_1)u_1.
$$
Taking $\lambda_1 = \recip{\mu_1}$, we have
$$
z = \lambda_1 y_1 + (1- \lambda_1) u_1,
$$
so that $z$ is an element of the open set
$$
W := \bigcup_{0 \leq \lambda < 1} \brac{\lambda y_1 + (1- \lambda) \cdot U}.
$$
By convexity $W \subset C$, hence $z \in \interior(C)$.
\end{proof}

\begin{lemma}\label{lem:ConvInt}
If $C \subset \R^n$ is convex then
$$
\interior(C) = \interior(\overline{C}).
$$
\end{lemma}

\begin{proof}
By Lemma \ref{lem:EmptyCase} we may assume that $\interior(C)$ is non-empty.  It suffices to show that if  $B_\eps(x) \subset \overline{C}$ for some $\eps > 0$ then $x \in \interior(C)$. Let $x_0 \in \interior(C)$.  Choosing $\delta > 0$ sufficiently small, one can ensure that
$$
y := x + \delta(x - x_0) \in B_\eps(x) \subset \overline{C}.
$$
Hence by the previous lemma $[x_0, y) \subset \interior(C)$.  Taking $\lambda = \frac{1}{1+\delta}$ we see that 
$$
x = (1-\lambda) x_0 + \lambda y \in (x_0, y) \subset \interior{C}.
$$
\end{proof}

In order to distinguish between the complex inner product on $\C^n$ and the real inner product on $\R^{2n}$, we write $\ang{x, y}$ for the former and $x \cdot y$ for the latter.

\begin{lemma}[Supporting hyperplane theorem]\label{SuppHypThm}  Let $C$ be a convex subset of $\R^n$ and $x \notin \interior{C}$.  Then there exists a non-zero vector $\phi \in \R^n \setminus \set{0}$ such that for all $y \in \overline{C}$ we have
$$
y \cdot \phi \leq x \cdot \phi .
$$
\end{lemma}

%
%

\begin{proof}  
Let us first prove the result under the assumption that $x \notin \overline{C}$.  The result is trivial if $C = \emptyset$, so we may assume that $C \neq \emptyset$.  Using absolute values to denote the $L^2$-norm on $\R^n$, it follows that there exists $y_0 \in \overline{C}$ such that
$$
|y_0-x| = \inf_{y \in \overline{C}} |y-x|.
$$

Heuristically, we expect that for any $y \in \overline{C}$, the angle formed in the plane between $x - y_0$ and $y-y_0$ should be obtuse.  If this angle were acute, then there should exist a point $y_1$ on the line segment between $y$ and $y_0$ such that $x$ is closer to $y_1$ than $y_0$ (draw a picture).  Since $\overline{C}$ is convex we have $y_1 \in \overline{C}$ and we have contradicted our choice of $y_0$.  

More rigourously, we show that for any $y \in \overline{C}$ we have
\begin{equation}\label{eqn:InnerProd}
\brac{x-y_0} \cdot \brac{y-y_0} \leq 0.
\end{equation}
Suppose this is not the case.  Then we claim that there exists $t \in (0,1]$ such that 
$$
|(1-t)y_0 + ty - x| < |y_0 - x|,
$$
and hence obtain our desired contradiction.  Write $(1-t)y_0 + ty - x = y_0-x + t(y-y_0)$, then square, expand out and divide through by $t$ to deduce that this is equivalent to the existence of $t \in (0, 1]$ such that 
$$
 t |y-y_0|^2 < 2\brac{x-y_0}\cdot\brac{y-y_0}.
$$
Since we are assuming that \eqref{eqn:InnerProd} does not hold, we may take
$$
t := \min\set{1, \frac{\brac{x-y_0}\cdot\brac{y-y_0}}{|y-y_0|^2}}.
$$

Assuming that $x \notin \overline{C}$, we may take $\phi := x - y_0 \neq 0$ to deduce that for any $y \in \overline{C}$ we have
$$
y\cdot\phi \leq {y_0\cdot \phi} \leq {x\cdot \phi},
$$ 
the latter inequality following from the fact that $(x\cdot \phi)-\brac{y_0\cdot \phi} = \brac{\phi\cdot \phi} \geq 0$.

It remains to prove the result when $x \in \overline{C} \setminus \interior(C)$.  Since $C$ is convex, it follows from Lemma \ref{lem:ConvInt} that $x \notin \interior(\overline{C})$, so that for any $m \in \N$ there exists $x_m \notin \overline{C}$ such that 
$$
|x - x_m| \leq 1/m.
$$
By our previous argument, there exists $\phi_m \neq 0$ such that for any $y \in \overline{C}$ we have
\begin{equation}\label{eqn:LimIneq}
y\cdot\phi_m  \leq {x_m\cdot \phi_m}.
\end{equation}
Normalising so that $|\phi_m| = 1$, we have a sequence in a compact set, so there exists a convergent subsequence $\phi_{k(m)} \to \phi$ with $|\phi| = 1$.  Taking limits in \eqref{eqn:LimIneq} then gives the desired inequality.
\end{proof}

\begin{corollary}[Complex supporting hyperplane theorem]\label{lem:CompSuppHyp}
Let $C$ be a convex subset of $\C^n$ and $x \notin \interior(C)$.  Then there exists $\phi \in \C^n \setminus \set{0}$ such that for all $y \in \overline{C}$ we have 
$$
\Realnobrac\ang{y, \phi }\leq \Realnobrac \ang{x,\phi}.
$$
\end{corollary}

\begin{proof}
This follows from the observation that for $x, y \in \C^n \cong \R^{2n}$ we have
$$
\Realnobrac\ang{x, y} = x \cdot y.
$$
\end{proof}


\section{The semi-finite minimax theorem}\label{minimax section}

We have not been able to find a reference for the variant of the minimax theorem employed in \S \ref{BoundedApproxSec}.

\begin{proposition}[Semi-finite minimax]\label{lem:minimax}
Let $A$ and $B$ be non-empty compact convex subsets of $\R^n$ at least one of which is equal to the convex hull of finitely many points.  Then there exist ${a_0} \in A$ and ${b_0} \in B$ such that for any $a \in A$ and any $b \in B$ we have
$$
a \cdot {b_0} \leq  {a_0} \cdot b.
$$
\end{proposition}

\begin{proof}
We may assume that $A$ is the convex hull of finitely many points, otherwise we re-label, taking $A' := -B$ and $B' := A$ to obtain $b_0 \in B$ and $a_0 \in A$ such that for any $b \in B$ and $a\in A$ we have
$$
-b \cdot {a_0} \leq  -{b_0} \cdot a,
$$
which yields the claimed result.

Define 
$$
L := \sup_{a \in A} \inf_{b \in B} (a\cdot b) \quad \text{and} \quad U := \inf_{b\in B} \sup_{a \in A}  (a\cdot b).
$$
In order to prove the proposition, it suffices to establish that
\begin{enumerate}[(i)]
\item There exists ${a_0} \in A$ and ${b_0} \in B$ such that
$$
L = \inf_{b \in B} ({a_0}\cdot b) \quad \text{and}\quad U =  \sup_{a \in A}  (a\cdot {b_0}).
$$
\item $U \leq L$.
\end{enumerate}
 
We begin by showing that 
\begin{equation}\label{MinimaxStep1}
- \infty < L \leq U < \infty.
\end{equation}
For any $a_1 \in A$ and $b_1 \in B$ we have
$$
\inf_{b \in B} (a_1 \cdot b) \leq a_1 \cdot b_1 \leq \sup_{a \in A} (a\cdot b_1).
$$
Since $a_1$ and $b_1$ are arbitrary, it follows that $L \leq U$.
Since $B$ is non-empty, there exists $b_1 \in B$.  Thus
$$
U \leq \sup_{a \in A} (a\cdot {b_1})
$$
By compactness, there exists $a_1 \in A$ such that 
$$
\sup_{a \in A} (a\cdot b_1) = a_1 \cdot b_1 < \infty.
$$
We conclude that 
$
 U < \infty.
$
Similarly, compactness of $B$ and non-emptiness of $A$ yields $L > - \infty$.  This establishes \eqref{MinimaxStep1}.

Since $U$ is finite, for any $k \in \N$ there exists $b_k \in B$ such that
$$
U \leq \sup_{a \in A} (a \cdot b_k) \leq U + \recip{k}.
$$
By compactness of $B$, there exists a convergent subsequence $b_{k_m} \to b \in B$. Continuity of the map $(a,b) \mapsto a\cdot b$ then ensures that for any $a \in A$ we have
$$
a \cdot b = \lim_{m\to \infty} a\cdot b_{k_m} \leq U.
$$
Thus $\sup_{a \in A} (a\cdot b) \leq U$, which by definition of $U$ implies that $\sup_{a \in A} (a\cdot b) = U$.  A similar argument holds for $L$.  This proves (i).

Finally, we show that for any $\alpha \in \R$ we either have $L \geq \alpha$ or $U \leq \alpha$.  Combining this with the fact that $L \leq U$, it follows that $L = U$ (if not, any $\alpha \in (L, U)$ leads to a contradiction).

Since $A$ is the convex hull of finitely many points, there exist $a_1, \dots, a_k \in A$ such that
$$
A = \set{ \sum_{i=1}^k \lambda_i a_i : \lambda_i \geq 0 \text{ and } \sum_i \lambda_i = 1}.
$$
Given $b \in B$ let us write
\begin{equation*}\label{eqn:v_bDefn}
v_b := \Bigbrac{(a_1\cdot b) - \alpha, \dots, (a_k \cdot b)- \alpha} \in \R^k.
\end{equation*}
Define $C$ to be the convex hull of the set
$$
\set{v_b : b \in B} \cup \set{e_1, \dots , e_k}.
$$

Let us first suppose that $0 \in C$.  Then there exist $b_1, \dots, b_m \in B$, $\lambda_1,\dots, \lambda_m, \mu_1, \dots, \mu_k \geq 0$ such that $1 = \sum_i \lambda_i + \sum_j \mu_j$ and
\begin{equation}\label{eqn:LinComb1}
0 = \sum_i \lambda_i v_{b_i} + (\mu_1, \dots, \mu_k).
\end{equation}
It follows that for each $j = 1, \dots, k$ we have
\begin{equation}\label{eqn:farkas1}
 \sum_{i=1}^m \lambda_i \bigbrac{(a_j\cdot b_i) - \alpha} \leq 0.
\end{equation}
By \eqref{eqn:LinComb1} we cannot have all $\lambda_i$ equal to zero.  We may therefore re-normalise, to conclude that there exist $\lambda_i \geq 0$ with $\sum_i \lambda_i = 1$ satisfying \eqref{eqn:farkas1}.  We deduce that for each $j = 1, \dots, k$ we have
$$
a_j\cdot \sum_{i=1}^m \lambda_i b_i \leq \alpha.
$$
Convexity then shows that for $b = \sum_i\lambda_ib_i \in B$ and for any $a \in A = \mathrm{ConvexHull}(a_1, \dots, a_k)$ we have $a \cdot b \leq \alpha$.  Hence $U \leq \alpha$.

Next suppose that $0 \notin C$.  By the supporting hyperplane theorem (Lemma \ref{SuppHypThm}, and the remark which follows it), there exists $\phi \in \R^k \setminus\set{0}$ such that for all $b_1, \dots, b_m \in B$ and $\lambda_1, \dots, \lambda_m, \mu_1, \dots, \mu_k \geq 0$ with $\sum_i \lambda_i + \sum_j \mu_j = 1$ we have
$$
\brac{\sum_i \lambda_i v_{b_i} +\sum_j \mu_j e_j} \cdot \phi \geq 0.
$$
In particular, we have
$$
\phi_j = e_j \cdot \phi \geq 0 \qquad (j = 1, \dots, k),
$$
and for each $b \in B$ we have
\begin{equation*}\label{eqn:farkas2}
v_b \cdot \phi \geq 0.
\end{equation*}
Since $\phi \neq 0$ we may re-normalise to conclude that there exists $\phi_j \geq 0$ with $\sum_j \phi_j =1$ such that for any $b \in B$ we have
$$
\Bigbrac{\sum_{j = 1}^k \phi_j a_j }\cdot b \geq \alpha.
$$  Convexity of $A$ then gives the existence of $a = \sum_j \phi_j a_j \in A$ such that for all $b \in B$ we have $a \cdot b \geq \alpha$, so that $L \geq \alpha$.  
\end{proof}

Recall that in order to distinguish between the complex inner product on $\C^n$ and the real inner product on $\R^{2n}$, we write $\ang{x, y}$ for the former and $x \cdot y$ for the latter.

\begin{corollary}[Complex minimax]\label{CompMini}
Let $A$ and $B$ be non-empty compact convex subsets of $\C^n$ at least one of which is equal to the convex hull of finitely many points.  Then there exist $a_0 \in A$ and ${b_0} \in B$ such that for any $a \in A$ and any $b \in B$ we have
$$
\Realnobrac\ang{a , {b_0}} \leq  \Realnobrac\ang{{a_0} , b}.
$$
\end{corollary}

\section{The Weierstrass polynomial approximation theorem}\label{Weier}
Given a real number $x$ write
$$
x_+ := \max\set{x , 0} = \trecip{2}(x + |x|).
$$
\begin{lemma}[Weierstrass polynomial approximation]
There exists an absolute constant $C >0$ such that for any $\eps \in (0,1)$ there exists a polynomial $P$ of degree at most $C\eps^{-2/3}$ and height at most $\exp(C\eps^{-2/3})$ such that 
$$
\sup_{|x| \leq 1} |P(x) - x_{+}| \leq \eps.
$$
\end{lemma}

\begin{proof}
By the Taylor series theorem, for any $t \in [0,1)$ we have
\begin{equation}\label{Taylor1}
(1-t)^{1/2} = - \sum_{n=0}^{N} c_n t^n +  O\brac{c_{N+1}},
\end{equation}
where 
$$
c_n = \frac{(2n)!}{(2n-1)2^{2n} (n!)^2}
$$
Using Stirling's formula, one can check that there exists a constant $C$ such that
$$
c_n \sim Cn^{-3/2} \quad \text{as $n \to \infty$}.
$$
In particular, by absolute convergence and continuity, the approximation \eqref{Taylor1} is valid for $t \in [0, 1]$.

For any $x \in [-1,1]$ we see that
$$
|x| = \bigbrac{ 1 - (1-x^2)}^{1/2} = \sum_{m=0}^N(-1)^m\brac{\sum_{n = m }^N c_n\binom{n}{m}} x^{2m} + O(N^{-3/2}).
$$
Using the crude bound
$$
\Bigabs{ \sum_{n = m }^N c_n\binom{n}{m}} \leq 2^N,
$$
we deduce that for any $N \in \N$ there exists a real polynomial $P_N$ of degree at most $2N$ and height at most $C\exp(N)$ such that
$$
\sup_{x \in [-1, 1]} |P_N(x) - |x|| \ll N^{-3/2}.
$$

The result now follows on taking $P(x) := \trecip{2}(P_N(x) + x)$ and ensuring that $N \geq C/\eps^{2/3}$ for some absolute constant $C$.
\end{proof}

\end{document}